\documentclass[11pt,twoside, a4paper, english, reqno]{amsart}
\usepackage{amssymb,amsfonts,latexsym,color}
\usepackage{graphics,verbatim}
\usepackage{graphicx}
\usepackage{a4wide}
\usepackage{hyperref}
\newtheorem{theorem}{Theorem}[section]

\newtheorem{lemma}{Lemma}[section]

\newtheorem{definition}{Definition}[section]

 \newcommand{\<}{\left\langle}
\renewcommand{\>}{\right\rangle}

\newcommand{\eps}{\varepsilon}

\newcommand{\be} {\begin{equation}}
\newcommand{\ee} {\end{equation}}
\newcommand{\bea} {\begin{eqnarray}}
\newcommand{\eea} {\end{eqnarray}}
\newcommand{\Bea} {\begin{eqnarray*}}
\newcommand{\Eea} {\end{eqnarray*}}
\newcommand{\pa} {\partial}

\newcommand{\al} {\alpha}
\newcommand{\ba} {\beta}
\newcommand{\de} {\delta}
\newcommand{\na}{\nabla}
\newcommand{\ga} {\gamma}

\newcommand{\Om} {\Omega}

\newcommand{\De} {\Delta}
\newcommand{\la} {\lambda}

\newcommand{\R}{\mathbb R}
\newcommand{\N}{\mathbb N}
\newcommand{\Rn}{\mathbb R^N}
\newcommand{\Iom}{\int_{\Omega}}
\newcommand{\deb}{\rightharpoonup}
\catcode`\@=11

\makeatletter \@addtoreset{equation}{section} \makeatother

\begin{document}
	\title[On the existence  of semilinear non-local elliptic systems ]{On the study of semilinear non-local elliptic systems }
	\thanks{The first author's research is funded by the Czech Science Foundation, project GJ19--14413Y. The second author's research is supported by Marie Sk{\l}odowska-Curie Individual Fellowships H2020-MSCA-IF-2019, P 888255. }
\author{Debangana Mukherjee}
\address{Department of Mathematics and Statistics, Masaryk University, 61137 Brno,  Czech Republic}

\email{mukherjeed@math.muni.cz, \,\, debangana18@gmail.com }

\author{Debopriya Mukherjee}
\address{Department of Mathematics and Information Technology, Montanuniversit\"at, Franz Josef Strasse, Leoben-8700, Austria}

\email{debopriya.mukherjee@unileoben.ac.at}

\maketitle
\date{}

\begin{abstract} 
	The purpose of this paper is to study the existence of solutions for
	 semilinear  elliptic system driven by fractional Laplacian
	and establish some new existence results  which are obtained by virtue of the local linking theorem and the saddle point theorem. To make the nonlinear scheme feasible, rigorous analysis of the function space involved and corresponding energy functional is necessary.

	
\end{abstract}


\section{Introduction}
A very captivating field in nonlinear analysis embraces the study of elliptic equations involving fractional operators. Freshly, tremendous thoughtfulness is given to such problems, in light of pure mathematical research and taking into account concrete real world applications. Certainly, such operators turn out naturally in different conditions, serving as several physical phenomena, population dynamics, mathematical finance, probability theory and many more. 
	In recent years, the existence  and multiplicity of solutions for elliptic systems have been extensively contemplated. 
	Precisely, for a quick review on the existence of nontrivial solutions for Laplace  systems, we refer \cite{Conti-Terracini-02,Costa-94, Drabek-01}. 
Existence of solutions and multiple solutions for some elliptic problems which involves the square root of the Laplacian with sign-changing weight is investigated in  \cite{Yu-12}.
In \cite{Zhang-09}, the authors have also obtained existence of weak solutions by using the technique of variational methods for a class of semilinear and quasilinear elliptic systems. 
A very delicate analysis is carried out in proving the existence results for some nonlinear ellipic systems
via analysis on Palais-Smale condition in \cite{Costa-94}.

	In past few years, several authors have studied the following type of semilinear and quasilinear elliptic problems,
\begin{equation}\label{pblm}
\begin{aligned}
\begin{cases}
&-\De u =f(x,u)\,\text{ in }\,\Om,\\
&u=0\,\text{ on }\, \pa \Om,
\end{cases}
\end{aligned}
\end{equation}
where $\Om$ is a bounded domain in $\Rn$. Here, $f$ is a nonlinear reaction term which
has been widely investigated about existence of solutions by using artistry tactics from nonlinear functional analysis, essentially using variational methods, degree theory, sub and supersolutions,
see \cite{Amann,Ambro,Amb-Rabin-73,Drabek-01}. A good amount of inquest and research has been 
accomplished in the study of $p$-Laplacian $\De_p u=\text{div}(|\na u|^{p-2} \na u)$
of (\ref*{pblm}), where $1<p<\infty$, see the book \cite{Drabek-97}
for more resource.

P. H. Rabinowitz in [Theorem 5.3, {\cite{Rabin-86}}] has established existence of weak solution by means of generalized Mountain Pass Theorem
for (\ref{pblm}) when $f(x,u)=\la a(x)u+g(x,u)$ in  a bounded domain $\Om \subset \Rn$ whose boundary is a smooth manifold under suitable assumptions on $a(x)$ and $g(x)$.
 Contemplating on elliptic systems in the classical case, Zou \cite{Zou-01} has augmented on the following problems;
\begin{equation*}
	\begin{cases}
	\begin{aligned}
	-\De u&= \la u\pm \de v + F_u(x,u,v) \,\text{ in }\,\Om,\\
		-\De v&= \de u\pm \ga v + F_v(x,u,v) \,\text{ in }\,\Om,\\
		u&=v=0\,\text{ on }\, \pa \Om,
	\end{aligned}	
\end{cases}
\end{equation*}	
and obtained existence of infintely many solutions with small energy assuming appropriate conditions on $F$. Zhang and Zhang in \cite{Zhang-09} have build up the existence result 
in pursuit of the system of (\ref{pblm})
employing minimax methods. 

In the non-local framework,
see some attributing references	\cite{Bhakta-3, Brasco-16,Shibo-16,Lindqvist,Squassina-16, Deba-tuhi} which has been carried through in assaying existence and multiplicity results,
under necessary assumptions of $f$. 
Multiplicity results for a class of non-local elliptic operators
by means of variational and topological methods using Morse theory is investigated in \cite{Shibo-16}.
 In \cite{Quaas-18}, the authors have demonstrated the existence result of atleast one positive solution for fractional Laplace system
whose proof heavily relies on the topological degree theory.
Very recently, in \cite{Ser-1}, a delving work to obtain existence of a non-trivial non-negative solution for non-local fractional operators is done using an iterative proficiency and a penalization method.

Motivated by the above literature, allowing for
 $\Om \subset \Rn$ to  be an open, smooth, bounded domain with smooth boundary, $s \in (0,1),N>2s$, $\la \in \R$,
 we investigate the following non-local semilinear elliptic system
\begin{equation*}
(\mathcal{P}_\la)
\left\{\begin{aligned}
(-\De)^s u  &=\la \big( f(x)u+g(x)v  \big) +F_u(x,u,v)   \quad\text{in }\quad \Om, \\
(-\De)^s v  &=\la \big( g(x)u+ h(x)v  \big) +F_v(x,u,v)   \quad\text{in }\quad \Om, \\
u =v=& 0  \quad\text{in }\quad \Rn \setminus \Om.
\end{aligned}
\right.
\end{equation*}
Here the non-local Operator $(-\Delta)^s$ is defined as follows:
	\begin{align} \label{frac-s}
	(-\Delta)^s u(x)=\lim_{\eps\to 0}\int_{\mathbb{R}^N\setminus B_\eps(x)}\frac{(u(y)-u(x))}{|x-y|^{N+2s}}dy,\,\,\,x\in\mathbb{R}^N,
	\end{align}
	and $F \in C^1(\bar{\Om} \times \R^2)$, $f,g,h \in C(\bar{\Om})$ with $F_u=\frac{\pa F}{\pa u}$, $F_v=\frac{\pa F}{\pa v}$.
	Our aim in this article is to find existence of weak solutions for $(\mathcal{P}_\la)$.
	In the classical case of the Laplacian $-\De$, the analogue of Theorems \ref{thm-1} and \ref{thm-2} are provided in \cite{Wu-NA}; in this context, these results are canonical but a natural extension of classical results to the non-local fractional framework.
	By virtue of the Laplacian, which is local  by nature, reckoning on non-local operators demand information throughout the whole domain; the virtue of our work lies in overcoming these difficulties appeared due to the fractional pattern. The aspect of our article lies on the ground about the study of fractional semilinear ellitpic system on a bounded domain in $\Om$ using variational methods, involving a careful survey on weighted eigenvalue problem in the fractional picture. As far as we know, the results accomplished in this manuscript are new and are not handy in literature.

	\textbf{Notation}. 
	Throughout this paper, we denote by $c, C ,C_i$ (for $i=1,2, \cdots$) the generic positive constants which may vary from line to line. We mark by
	$|u|_p:=\left(\int_\Om |u(x)|^p dx\right)^\frac{1}{p}$, the $L^p$ norm in $\Om$.
	Corresponding to any $r>1$, we signify by $r'$, the conjugate of $r$, which is $
	\frac{r}{r-1}$. 

\subsection{Hypothesis}We introduce the following conditions on the parameters and functions involved in the system $(\mathcal{P}_\la)$.
\begin{itemize}
	\item [\textit{$(\bf H_1)$}]
	$g(x) \geq 0$ for all $x \in \bar{\Om}$.
	\item [\textit{$(\bf H_2)$}]
	$\max_{x \in \bar{\Om}} \max \{ f(x), h(x)  \} \geq 0$.	
	\item [\textit{$(\bf H_3)$}]
	There exists $C_1>0$ and $2<q<2_s^*$ such that
	\begin{equation*}
	  \big| F_u(x,u,v) \big| +\big|F_v(x,u,v)\big| \leq 	C_1 \big(1+ |u|^{q-1}+|v|^{q-1} \big),
	\end{equation*}
	for all $(x, u, v) \in (\Om \times\R^2 )$ where 
	\begin{gather*}
	2^*_s=\begin{cases}
	\frac{2N}{N-2s},\, 2s<N,\\
	+\infty,\, 2s \geq N.
	\end{cases}
	\end{gather*}
		\item [\textit{$(\bf H_4)$}]
		There exists $a \in L^{\infty}(\Om)$ such that
		\begin{gather*}
		\limsup_{|(u,v)| \to \infty} \frac{|F_u(x,u,v)|+|F_v(x,u,v)|}{ \big( |u|^2+|v|^2  \big)^{\frac{1}{2}}} \leq a(x) ,\,\text{uniformly for a.e.}\, x \in \Om. 
		\end{gather*}
			\item [\textit{$(\bf H_5)$}]
			There holds
			\begin{equation*}
		\lim_{|(u,v)| \to 0}	\frac{|F_u(x,u,v)|+|F_v(x,u,v)|}{\big(|u|^2+|v|^2\big)^{\frac{1}{2}}}=0 \,\text{uniformly for a.e.}\, x \in \Om. 
		\end{equation*}
			\item [\textit{$(\bf H_6)$}]
			There holds
				\begin{equation*}
			\lim_{|(u,v)| \to 0}	\frac{F(x,u,v)}{|u|^2+|v|^2}=\infty \,\text{uniformly for a.e.}\, x \in \Om. 
			\end{equation*}
				\item [\textit{$(\bf H_7)$}] There exists
			\begin{equation*}
			p \begin{cases}
			\geq \frac{2N}{N+2s}(q-1) \, N>2s,\\
			>q-1, \, N \leq 2s,
			\end{cases}
			\end{equation*}	
				such that
			\begin{equation*}
			\liminf_{|(u,v)| \to \infty} \frac{F_u(x,u,v)u+F_v(x,u,v)v-2F(x,u,v) }{\big( |u|^2+|v|^2 \big)^{\frac{p}{2}}},
			\end{equation*}	
			uniformly for a.e. $x \in \Om$.
			\item [\textit{$(\bf H_8)$}]	
			There exists constant 
			\begin{equation*}
			\ga \begin{cases}
			\geq \frac{2N}{N+2s}, N>2s,\\
			>1, N \leq 2s,
			\end{cases}
			\end{equation*}
			such that 
    \begin{equation*}
    \limsup_{(u,v) \to \infty} \frac{F_u(x,u,v)u+F_v(x,u,v)v-2F(x,u,v) }{\big( |u|^2+|v|^2 \big)^{\frac{\ga}{2}}},
    \end{equation*}			
				uniformly for a.e. $x \in \Om$.
			\item [\textit{$(\bf H_9)$}]		
			There holds
			\begin{equation*}
		\liminf_{|(u,v)| \to \infty} \frac{F_u(x,u,v)u+F_v(x,u,v)v}{\big( |u|^2+|v|^2 \big)} \geq 0,
		\end{equation*}	
		uniformly for a.e. $x \in \Om$.		
\end{itemize}

\subsection{Main results}

To use variational methods for semilinear elliptic system $(\mathcal{P}_\la)$, we first study non-local weighted eigenvalue problem
\begin{equation*}
	(\mathcal{E}_\la)
\left\{\begin{aligned}
(-\De)^s u  &=\la \big( f(x)u+g(x)v  \big)  \quad\text{in }\quad \Om, \\
(-\De)^s v  &=\la\big( g(x)u+ h(x)v\big)    \quad\text{in }\quad \Om, \\
u =v=& 0  \quad\text{in }\quad \Rn \setminus \Om.
\end{aligned}
\right.
\end{equation*}
With the help of Hypotheses $(H_1)-(H_2)$ and using the spectral theory of symmetric compact operators, we infer that there exists a sequence of eigenvalues $0<\la_1<\la_2 \leq \cdots \leq \la_k \to \infty$ of $(\mathcal{E}_\la)$.
 We denote the normalized eigenfunction corresponding to the eigenvalue $\la_1$ by $(\phi_1,\psi_1)$. For technical analysis, we choose $\phi_1>0$ in $\Om$ and $\psi_1>0$ in $\Om$. 
 
 In the first outcome of the manuscript, we demonstrate existence of non-trivial weak solution of the system $(\mathcal{P}_\la)$ under some postulates  for $\la$ varying in consecutive distinct eigenvalues.
 \begin{theorem}\label{thm-1}
 Let us assume that Hypotheses 	$(H_1)$-$(H_3)$, $(H_5)$-$(H_7)$ be satisfied. Then, for every $k \in \N$ 
  and for every $\la \in (\la_k, \la_{k+1}),$ provided $\la_k \neq \la_{k+1}$,  the system $(\mathcal{P}_\la)$ has one non-trivial weak solution.
 \end{theorem}
	
 We expound briefly the proof of Theorem \ref{thm-1}. To begin with, we will find the existence of a weak solution as a critical point of a suitable $C^1$ energy functional, defined on an appropriately chosen Banach space $X$. As a next step, we show that the associated energy functional has a local linking at zero with respect to a direct sum decomposition pair of Banach spaces of $X$. We further ensue that the functional satisfies (PS)$^*$ condition (see Definition \ref{pscondition}) and maps bounded sets into bounded sets; subsequently showing the corresponding functional is anticoercive (see Lemma \ref{coercive.1}). Our strategy in proving Theorem \ref{thm-1} relies on Theorem \ref{assthm.1}.

In the next harvest, we exhibit the existence of weak solution of the system $(\mathcal{P}_{\la)}$ for $\la=\la_1$ under some attributions together with the assumption that the set $\{ x: a(x) <\la_2-\la_1 \big\} $ has positive Lebesgue measure and $g \equiv 0, \{f, h \}\geq 1$ over ${\bar{\Om}}$.

\begin{theorem}\label{thm-2}
	Let us assume that Hypotheses $(H_4),\,\,(H_8)$ and $(H_9)$ be satisfied. Furthermore, assume that $ a(x) \leq \la_2-\la_1$ a.e. $x \in \Om$ with the condition that $\mathcal{L}^N \big( \big\{ x: a(x) <\la_2-\la_1 \big\}  \big)>0,$ where $\mathcal{L}^N$ denotes the $N$-dimensional Lebesgue measure. Then, the system $(\mathcal{P}_{\la_1)}$ has a weak solution, provided $g \equiv 0, \{f(x), h(x) \}\geq 1$, for all ${x \in \bar{\Om}}$.
\end{theorem}	
Our approach in proving Theorem \ref{thm-2} is based on Theorem \ref{assthm.2}.
 In the primitive step, we rewrite the functional as sum of two functionals, first one of which is linear and involves bounded, self-adjoint operators, subsequently decomposing the space $X$ as an orthogonal pair of subspaces, resulting that the given functional is coercive in one subspace of these pairs (see Lemma \ref{coercive.2}) and anticoercive in another (see Lemma \ref{anti_coercive}), proximating that the correlated energy functional satisfies Cerami condition (see Definition \ref{Cb}).
	
		\textbf{Organisation of the paper}.
		The present paper is organized as follows. In Section \ref{prel}, we recall some preliminary results, covering the function spaces. We mention a few abstract critical point theorems  in Section \ref{abstract}, vital to prove our main results. Section \ref{weak-sol} consists of introducing the variational method to establish existence of weak solution.  We provide the proof of Theorem \ref{thm-1} in Section \ref{Thm-1} and the proof of Theorem \ref{thm-2} in Section \ref{Thm-2}.

\section{Preliminaries}\label{prel}
In this section we define appropriate function spaces which are required for our analysis. Let $ s\in(0,1),\, N>2s,\, 2_s^*:=\frac{2N}{N-2s}.$ We denote the standard fractional Sobolev space by $H^s(\Omega)=W^{s,2}(\Omega)$ endowed with the norm
$$
\|{u}\|_{H^s(\Om)}:=\|{u}\|_{L^2(\Om)}+\left(\int_{\Om\times\Om} \frac{|u(x)-u(y)|^2}{|x-y|^{N+2s}}dxdy\right)^{1/2}.
$$
We set $Q:=\R^{2N}\setminus (\Om^c \times \Om^c)$, where $\Om^c=\Rn \setminus \Om$ and define $$
X_s(\Om):=\Big\{u:\mathbb{R}^N\to\mathbb{R}\mbox{ measurable }\Big|u|_{\Omega}\in L^2(\Omega)\mbox{ and }
\int_{Q} \frac{|u(x)-u(y)|^2}{|x-y|^{N+2s}}dxdy<\infty\Big\}.
$$
The space $X_s(\Om)$ is endowed with the norm defined as
$$\|u\|_s:=|u|_2+\left(\int_{Q} \frac{|u(x)-u(y)|^2}{|x-y|^{N+2s}}dxdy\right)^{1/2},$$
where $|u|_2:=\left(\int_\Om |u(x)|^2 dx\right)^\frac{1}{2}$. We note that in general $H^s(\Om)$ is not same as $X_s(\Om)$ as $\Om\times\Om$ is strictly contained in $Q$.
We define the space $X_{0,s}(\Om)$ as
$$X_{0,s}(\Om) :=\Big\{u \in X_s(\Om) : u=0 \quad\text{a.e. in}\quad \Rn \setminus \Om\Big\} $$
or equivalently
as $\overline{C_0^\infty(\Om)}^{X_s(\Om)}$. It is well-known that for $X_{0,s}(\Om)$ is a Hilbert space endowed with the inner product
$$\langle u, v\rangle_{0,s}=\int_{Q} \frac{(u(x)-u(y))(v(x)-v(y))}{|x-y|^{N+2s}}\,dxdy,$$	
and the corresponding norm is given by
$$\|u\|_{0,s}=\left(\int_{Q} \frac{|u(x)-u(y)|^2}{|x-y|^{N+2s}}dxdy\right)^{1/2}.$$
Since $u=0$ in $\Rn\setminus\Om,$ the above integral can be extended to all of $\mathbb{R}^N.$ The embedding
$X_{0,s}(\Om)\hookrightarrow L^r(\Om)$ is continuous for any $r\in[1,2^*_s]$ and compact for $r\in[1,2^*_s).$ 


\section{Abstract Results}\label{abstract}
In this section, we recollect a few abstract critical point theorems. The resulting concepts are 
borrowed from \cite{Liwilem} and \cite{Silva-95}. 
Let $X$ be a real Banach space having a direct sum decomposition $X=X^1 \oplus X^2$. We deal with two sequences of subspaces 
$$ X_0^1 \subset X_1^1 \subset \cdots X^1,\, X_0^2 \subset X_1^2 \subset \cdots \subset X^2
$$
including $X^j=\overline{\cup_{n \in \N}}X_n^j$, $j=1,2$ and $\text{dim }(X_n^j)<\infty$, $j =1,2, n \in \N$.
We indicate $X_\al$ as $X_{\al_1} \oplus X_{\al_2}$
for every miulti-index $\al=(\al_1,\al_2) \in \N \times \N$. We revive that 
$$ \al \leq \ba \Leftrightarrow \al_1 \leq \ba_1,\text{ and } \al_2 \leq \ba_2.
$$
A sequence $\{\al_n \} \subset \N \times \N$ is admissible if, for every $\al \in \N \times \N$, there exists $n_0 \in \N$ such that $ n \geq n_0 \Rightarrow \al_n \geq \al.$ We announce by $f_\al$, $f$ restricted to $X_\al$ for every $f: X \to \R$.

We state now some essential definitions and theorems which are pivotal tools to obtain our main results.
\begin{definition}[$(PS)^*$ condition \cite{Liwilem}]\label{pscondition}
A functional $J\in C^1(X, \mathbb{R})$ is said to satisfy the $(PS)^*$ condition if every sequence $\{u_{\beta_n}\}\subset X$ with $\{\beta_n\}$ is admissible, $u_{\beta_n}\in X_{\beta_n}$ and
$\sup_n J(u_{\beta_n} ) < \infty$ alongside $J'_{\beta_n}( u_{\beta_n} )\to0$ possesses a convergent subsequence which converges to a critical point of $J.$
\end{definition}
\begin{definition}[Local linking \cite{Liwilem}]\label{loclink}
 Let $X$ be a Banach space with a direct sum decomposition $X = X_1\oplus X_2 .$ The functional $J\in C^1(X, \mathbb{R} )$
has a local linking at zero with respect to $( X_1,X_2)$ if there is $t > 0$ such that
\begin{align}\label{geqcon}
 J( u ) \geq 0,\,\quad \mbox{for all}\,\,u \in X_1\quad \mbox{with}\quad \|u\| \leq  t
\end{align}
and
\begin{align}\label{leqcon}
 J ( u ) \leq 0,\,\quad \mbox{for all}\,\,u \in X_2\quad \mbox{with}\quad \|u\| \leq  t.
\end{align}
\end{definition}
\begin{definition}[Cerami $b$ or $\,(C_b)$ condition]\label{Cb}
 Let $b\in\mathbb{R}$ and $J\in C^1(X,\mathbb{R})$ satisfies the $C_b$ condition, if for every sequence 
 $\{u_n\}\subset X$ with $J(u_n)\to b$ and $(1+\|u_n\|)J'(u_n)\to 0,$ as $n\to\infty,$
 possesses a convergent subsequence.
\end{definition}
\begin{theorem}[Theorem 1, \cite{Liwilem}]\label{assthm.1}
Let $X$ be a Banach space with a direct sum decomposition $X = X_1\oplus X_2.$ Assume that $J\in C^1 ( X , \mathbb{R} )$
satisfies the following conditions:
 \begin{itemize}
  \item [(a)]
  $J$ has a local linking at zero with respect to $( X_1 , X_2 ) ;$
  \item [(b)]
  $J$ satisfies the $(PS)^*$ condition;
  \item[(c)]
  $J$ maps bounded sets into bounded sets;
  \item[(d)]
  for every $k \in \mathbb{N},\,\, J ( u )\to-\infty$ as $u \in X_{1,k} \oplus X_2$ and $\| u\|\to\infty.$
 \end{itemize}
Then, $J$ has a nontrivial critical point.
\end{theorem}
\begin{theorem}[Theorem 1.1, \cite{Silva-95}]\label{assthm.2}
Let $X$ be a Hilbert space with a orthogonal direct sum decomposition $X = X_1\oplus X_2.$ Assume that $J\in C^1 ( X , \mathbb{R} )$
satisfies the following conditions:
 \begin{itemize}
  \item [(a)]
  $J$ can be written as
  $$
  J(u)=\frac{1}{2}\left\langle S(u),u\right\rangle+ K(u),\quad u\in X,
  $$
  with $S(u)=S_1\Pi_1(u)+S_2\Pi_2(u),\,\,u\in X,$ where $S_i:X_i\to X_i$ are bounded self-adjoint operator and $\Pi_i:X\to X_i$ are orthogonal projections for $i=1,2$ and $K'$ is compact;
  \item [(b)]
  there is $t_1\in\mathbb{R}$ such that
\begin{align}\label{leqcon.1}
 J ( u ) \leq t_1,\,\quad \mbox{for all}\,\,u \in X_1;
\end{align}
  \item[(c)]
  there is $t_2 \leq t_1$ such that
\begin{align}\label{geqcon.1}
 J ( u ) \geq t_2,\,\quad \mbox{for all}\,\,u \in X_2;
\end{align}
  \item[(d)]
  $J$ satisfies $C_b$ conditions.
 \end{itemize}
Then, $J$ has a critical value $t_0\in [t_2,t_1]$.
\end{theorem}

\section{Existence of weak solution}\label{weak-sol}
With the objective to obtain existence of weak solutions on relevant function space, we first consider the Hilbert space $X=X_{0,s}(\Om) \times X_{0,s}(\Om)$ with the norm
\begin{equation}\label{eq:norm}
\|(u,v)\|_X=\big( \|u\|_{0,s}^2+\|v\|_{0,s}^2 \big)^{\frac{1}{2}}, (u,v) \in X,
\end{equation}
and the corresponding inner product is denoted and defined by
\begin{gather*}
\< (u,v), (w,z)  \>_X =\<u,w\>_{0,s}+\<v,z\>_{0,s}\, (u,v),(w,z) \in X.
\end{gather*}	
The well-known fractional Sobolev embedding continues to hold for the space $X$, that is,
\begin{gather}\label{Sob-embed}
X \hookrightarrow L^r(\Om) \times L^r(\Om) \, \text{continuously for}\, 1 \leq r \leq 2^*_s,
\end{gather}	
and fractional compact embedding
\begin{gather}\label{Cpt-embed}
	X \hookrightarrow L^r(\Om) \times L^r(\Om) \, \text{compactly for}\, 1 \leq r < 2^*_s.
\end{gather}	
We recall that (\ref{Sob-embed}) is not true for $r=\infty$, that is, for $N=2s$. Let us denote the dual space of $X$ by $X'$. We consider the operator $L: X \to X'$ given by
\begin{equation}\label{eq:L}
\left\langle L(u,v),(w,z)\right\rangle=\Iom \big[ \big(f(x)u+g(x)v \big) w+\big( g(x)u+h(x)v \big)z  \big]\, dx, \, \text{for all}\, (u,v), (w,z) \in X.
\end{equation}
We certify that $L$ is well-defined, bounded linear operator. Indeed, by using H\"older's inequality and the fact that $f,g,h \in C(\bar{\Om})$, it follows,
\begin{equation*}
\begin{aligned}
\big| \<L(u,v),(w,z) \>\big|& \leq 2 \big( |f|_{\infty}+|g|_{\infty}+|h|_{\infty} \big) \big( |u|_2|w|_2+|v|_2|w|_2+|u|_2|z|_2+|v|_2|z|_2   \big)\\
&\leq 2C^2 \big( |f|_{\infty}+|g|_{\infty}+|h|_{\infty} \big)  \big( \|u\|_{0,s}+\|v\|_{0,s} \big)\big( \|w\|_{0,s}+\|z\|_{0,s}\big)\\
&\leq 2C^2 \big( |f|_{\infty}+|g|_{\infty}+|h|_{\infty} \big) \| (u,v) \|_{X} \|(w,z)\|_{X}
\\	
\end{aligned}	
\end{equation*}
for $(u,v), (w,z) \in X$. This implies,
\begin{gather*}
	\| L(u,v)\|_{X'} \leq 2C^2 \big( |f|_{\infty}+|g|_{\infty}+|h|_{\infty} \big)  \|(u,v)\|_X.
\end{gather*}	
This yields, $L$ is well-defined, bounded linear operator, making out that $L$ is symmetric. By using hypothesis $(H_1)-(H_2)$ and compactness of the embedding that $X \hookrightarrow L^2(\Om) \times L^2(\Om)$, we infer that $L$ is compact linear operator. Hence, the eigenvalue problem $(\mathcal{E}_\lambda)$ can be cast as:
\begin{gather}\label{eq:L-}
\big( (-\De)^s u, (-\De)^s v \big)=\la L(u,v),\, (u,v) \in X.	
\end{gather}	
Let us denote $\mathcal{D}=\begin{bmatrix}
(-\De)^s & 0\\
0 & (-\De)^s
\end{bmatrix}$.
Consequently, there exists a sequence of eigenvalues $0<\la_1<\la_2 \leq \la_3 \leq \cdots \leq \la_k \leq \cdots$
for the eigenvalue problem (\ref{eq:L-}) such that $\lim_{k \to \infty}\la_k=\infty$ and $\la_1$ is simple, positive and isolated in the spectrum $\mathcal{D}$
in $X$ and can be characterized by
\begin{gather}\label{eq:la_1}
	\frac{1}{\la_1}=\sup \big\{ \<L(u,v), (u,v)\> : \|(u,v)\|_X=1 \big\}.
\end{gather}	
Let $(\phi_1,\psi_1)$ be the normalized equation function corresponding to $\la_1$ such that $\phi_1>0$ and $\psi_1>0$ in $\Om$. For $k \in \N$, let us denote the eigenspace corresponding to the eigenvalue $\la_k$ by $V_k$ and
$W_k:=\oplus_{i=1}^k V_i$ and monitor using Raleigh quotient, the following variational inequality
\begin{equation}\label{eq:W_k}
\| (u,v) \|_X^2 \leq \la_k \< L(u,v), (u,v)\> \,\text{for all}\, (u,v) \in W_k
\end{equation}
and
\begin{equation}\label{eq:W_k-}
\| (u,v) \|_X^2 \geq \la_{k+1} \< L(u,v), (u,v)\> \,\text{for all}\, (u,v) \in W_k^{\perp},
\end{equation}
holds.
\section{Proof of Theorem \ref*{thm-1}}\label{Thm-1}
For $\la \in (\la_k,\la_{k+1})$,
let us denote the energy functional corresponding to the problem $(\mathcal{P}_\la)$ as:
$J_\la: X \to \R$ by
\begin{equation}\label{eq:J}
J_\la(u,v)=\frac{1}{2}\| (u,v)\|_X^2-\frac{\la}{2}\<L(u,v), (u,v)\>-\Iom F(x,u,v)\,dx \,\text{for all}\, (u,v) \in X,
\end{equation}
where $L$ is defined in (\ref{eq:L}). In order to bring out the proof of Theorem \ref{thm-1}, we need  the following Lemmas.
\begin{lemma}\label{c1}
Let $J_\lambda$ be defined above in \eqref{eq:J}. Then,
 $J_\la \in C^1(X;\R)$.
\end{lemma}
\begin{proof}
In view of the fact that $F$ satisfies \textit{$(H_3)$}, using Dominated Convergence Theorem and technique of differentiation under integration, one can show that the functional 
$(u,v) \mapsto \Iom F(x,u(x),v(x)) \, dx$ is continuously differentiable. As $L$ is linear and $f,g,h \in C(\bar{\Om})$, 
$$ (u,v) \mapsto \<L(u,v), (u,v)\>=\Iom  \big(( f(x)u+g(x)v)u+ (g(x)u+h(x)v)v)) \big)\, dx
$$
is also continuously differentiable. Hence, $J_\la$ is continuously differentiable  and
\begin{equation}\label{eq:J'}
\begin{aligned}
&\<J'_\la(u,v), (\phi,\psi)\>\\
&=\<u,\phi\>_{0,s}+\<v,\psi\>_{0,s}-\la \Iom  \big(( f(x)u+g(x)v)\phi+ (g(x)u+h(x)v)\psi)) \big)\, dx\\
&\quad-\Iom \big(F_u(x,u,v)\phi+F(x,u,v)\psi \big)\,dx\\
&=\int_{\R^{2N}} \frac{(u(x)-u(y)) (\phi(x)-\phi(y))}{|x-y|^{N+2s}}\, dxdy+\int_{\R^{2N}} \frac{(v(x)-v(y)) (\psi(x)-\psi(y))}{|x-y|^{N+2s}}\, dxdy\\
 &\quad-\la \Iom  \big(( f(x)u+g(x)v)\phi+ (g(x)u+h(x)v)\psi)) \big)\, dx
 -\Iom \big(F_u(x,u,v)\phi+F(x,u,v)\psi \big)\,dx,
 \end{aligned}
\end{equation}
for all $(\phi,\psi) \in X$. This finishes the lemma.
\end{proof}
From equation (\ref{eq:J'}), we observe that $(u,v)$ is a weak solution of $(\mathcal{P}_\la)$ iff $(u,v)$ is a critical point of the functional $J_\la$. For fixed $k \in \N$, let us denote
\begin{gather}\label{eq:X_i-1} 
X^1:=W_k^{\perp}\,\text{  and  }\, X^2:=W_k
\end{gather}
and accordingly we define for all $n\in\mathbb{N},$ 
\begin{equation}\label{eq:X_i-2}
X_n^1:=V_{k+1} \oplus V_{k+2} \oplus \cdots \oplus V_{k+n} \,\,\text{ and }\,\, X_n^2:=X^2.
\end{equation}
Then, we note that,
\begin{equation}\label{eq:X_i-3} 
	\begin{cases}
	X_1^1 \subset X_2^1 \subset \cdots \subset X_n^1 \cdots \subset X' \,\text{and}\,\, X^1=\overline{\cup_{i \in \N} X_i^1},  \\
	X_1	^2=X_2^2=\cdots =X_n^2=\cdots=X^2.
	\end{cases}	
\end{equation}	

\begin{lemma}\label{loclink}
Let $J_\lambda$ be defined above in \eqref{eq:J} and $(X^1, X^2)$ be defined above in \eqref{eq:X_i-1}-\eqref{eq:X_i-2}. Then, $J_\la$ has a local linking at zero with respect to $(X^1, X^2).$ 
\end{lemma}
\begin{proof}
	For simplicity, we assume $F(x,0,0)=0$. If not, one can take
\begin{gather*}
\tilde{F}(x,u,v)=F(x,u,v)-F(x,0,0),\,\text{for all}\, (u,v) \in \R^2,	
\end{gather*}	
in place of $F(x,u,v)$. Using the hypothesis \textit{$(H_3)$} and Young's inequality, we obtain,
\begin{equation*}
\begin{aligned}
	|F(x,u,v)|&=\bigg| \int_0^1 \frac{d}{dt} (F(x,tu,tv))\,dt\bigg|\\
	&=\bigg|\int_0^1 \big( F_u(x,tu,tv)u+F_v(x,tu,tv)v \big)\,dt  \bigg|\\
	&\leq \int_0^1 \big( (|F_u(x,tu,tv)|+|F_v(x,tu,tv)|)(|u|+|v|) \big) \, dt\\
	&\leq C_1 \big(1+|u|^{q-1}+|v|^{q-1} \big) \big( |u|+|v| \big)\\
	&\leq C_1 \big( |u|+|v|+|u|^q+|u||v|^{q-1}+|u|^{q-1}|v|+|v|^q  \big)\\
	&\leq C_1 \big( |u|+|v|+|u|^q+\frac{1}{q}|u|^q+\frac{q-1}{q}|v|^{q}+\frac{q-1}{q}|u|^q+\frac{1}{q}|v|^q+|v|^q   \big).
\end{aligned}	
 \end{equation*}
This makes evident the following,
\begin{gather}\label{eq:F}
|F(x,u,v)| \leq 2C_1 \big( |u|+|v|+|u|^q+|v|^q \big) \, \text{for all}\, (u,v) \in \R^2 \,\text{and}\, \text{ for all } x \in \Om. 	
\end{gather}
For given $\eps>0$ and using hypothesis \textit{$(H_5)$}, there exists $\de_1>0$ such that 
\begin{equation}\label{eq:F_u,v}
|F_u(x,u,v)|+|F_v(x,u,v)| \leq \eps \big( |u|^2+|v|^2 \big) \,\text{for all}\, ( |u|^2 +|v|^2)^{\frac{1}{2}} \leq \de_1\, \text{a.e.}\, x \in \Om.
\end{equation}
	Consequently, we arrive at the fact that,
\begin{equation*}
\begin{aligned}
| F(x,u,v)| & \leq \int_0^1 \big(|F_u(x,tu,tv) |+|F_v(x,tu,tv)|  \big)\big(|u|+|v| \big)\,dt\\
&\leq \eps \int_0^1 \big( |tu|^2+|tv|^2 \big)^{\frac{1}{2}}\big( |u|+|v| \big)\, dt
	\eps \big( |u|^2+|v|^2 \big),
\end{aligned}
\end{equation*}			
	which implies,
	\begin{gather}\label{eq:F-}
	\big| F(x,u,v)\big| \leq \eps ( |u|^2+|v|^2 ) \,\text{for all}\, (u,v) \in \R^2 \,\text{with}\, \big(|u|^2+|v|^2 \big)^{\frac{1}{2}} \leq \de_1, \text{a.e.}\, x \in \Om. 
	\end{gather}
Bringing together (\ref{eq:F_u,v}) and (\ref{eq:F-}), we figure out the following,
\begin{gather}\label{eq:F--}
|F(x,u,v)| \leq \eps \big( |u|^2+|v|^2 \big)+C_\eps \big( |u|^q+|v|^q  \big)\,\text{for all}\, (u,v) \in \R^2 \,\text{and a.e.}\, x \in \Om.	
\end{gather}	
To see this, for $\de_1< {(|u|^2+|v|^2)^{\frac{1}{2}}}$, we have,
$$1< \frac{ (|u|^2+|v|^2)^{\frac{1}{2}}}{\de_1}\leq \frac{|u|+|v|}{\de_1}.
$$
Adding up the fact $q>1$ results to,
\begin{gather}\label{eq:u,v}
|u|+|v| \leq \big(\frac{2}{\de_1}\big)^{q-1} ( |u|^q+|v|^q ).	
\end{gather}	
Hence, from (\ref{eq:F}), using (\ref{eq:u,v}),we analyze  for $(|u|^2+|v|^2 )^{\frac{1}{2}}>\de_1$, 
\begin{gather}\label{eq:u,v-}
| F(x,u,v)| \leq C_{\de_1} (|u|^q+|v|^q ).	
\end{gather}	
Combining (\ref{eq:F-}) and (\ref{eq:u,v-}), we obtain (\ref{eq:F--}). Using (\ref{eq:F--})  and the embedding (\ref{Sob-embed}), we see that, for $(u,v) \in X$, 
\begin{equation}\label{eq:F^}
\begin{aligned}
\Iom F(x,u(x),v(x))\,dx&\leq \eps \Iom \big( |u|^2+|v|^2 \big)\, dx + C_\eps \Iom \big( |u|^q+|v|^q \big)\, dx\\
&\leq C^2 \big( \|u\|_{0,s}^2+\|v\|_{0,s}^2 \big)+ C_{\eps,q}\big( \|u\|_{0,s}^q+\|v\|_{0,s}^q \big)\\
&\leq \eps C \|(u,v)\|_X^2+C_{\eps,q} \|(u,v)\|_X^q.
\end{aligned}
\end{equation}
The above inequality together with (\ref{eq:W_k-}) gives us for $(u,v) \in X^1=W_k^\perp$ 
\begin{equation*}
	\begin{aligned}
J_\la(u,v)&\geq \frac{1}{2}\|(u,v)\|_X^2-\frac{\la}{2\la_{k+1}}\| (u,v)\|_X^2-\Iom F(x,u,v)\,dx\\
&\geq \big( \frac{1}{2}-\frac{\la}{2\la_{k+1}}-\eps C^2 \big) \|(u,v)\|_X^2-C_{\eps,q}\|(u,v)\|_X^q.\\
	\end{aligned}
\end{equation*}	
As $\la<\la_{k+1}$ and $q>2$, we can choose $\eps>0$ small so that $ \big( \frac{1}{2}-\frac{\la}{2\la_{k+1}}-\eps C^2 \big)>0$ and then there exists $R_1>0$ such that
\begin{equation}\label{eq:J-la}
J_\la(u,v) \geq 0 \, \text{for all}\, (u,v) \in X^1 \,\text{with}\, \|(u,v)\|_X \leq R_1.
\end{equation}
In a similar fashion, using (\ref{eq:W_k}) and (\ref{eq:F^}), we get for $(u,v) \in X^2=W_k$ 
\begin{gather*}
J_\la(u,v)\leq \big(\frac{1}{2}-\frac{\la}{2\la_k}+\eps C^2 \big)\|(u,v)\|_X^2+C_{\eps,q}\|(u,v)\|_X^q.
\end{gather*}	
Since $\la>\la_k$ and $q>2$, we can choose $\eps>0$ small enough so that $\big( \frac{1}{2}-\frac{\la}{2\la_k}+\eps C^2  \big)<0$ and there exists $R_2>0$ such that for $(u,v) \in X^2$ with $\|(u,v)\|_X \leq R_2$, 
\begin{gather}\label{J}
	J_\la(u,v)<0.
\end{gather}	
Taking $R=\min \{ R_1,R_2 \}$, we see from (\ref{eq:J-la}) and (\ref{J}) 
\begin{equation*}
	J_\la(u,v)
\begin{cases}
	\geq 0 \,\text{ for all }\, (u,v) \in X^1 \, \text{ with }\, \|(u,v)\|_X \leq R;\\
		\leq 0 \,\text{ for all }\, (u,v) \in X^2 \, \text{ with }\, \|(u,v)\|_X \leq R.
\end{cases}		
\end{equation*}	
Thus, we infer  that $J_\la$ has a local linking at zero with respect to $(X^1, X^2)$. This completes the proof.
\end{proof}
\begin{lemma}\label{bbd}
Let $J_\lambda$ be defined above in \eqref{eq:J}. Then,
$J_\la$ maps bounded sets of $X$ into bounded sets of $\R$.
\end{lemma}
\begin{proof}
Using the definition of $\la_1$ and (\ref{eq:F^}) we deduce
\begin{equation*}
J_\la(u,v) \leq \big( \frac{1}{2}+\frac{\la}{2\la_1} \big) \| (u,v)\|_X^2+\eps C^2 \|(u,v)\|_X^2+C_{\eps,q}\|(u,v)\|_X^q.
\end{equation*}
Hence, if $\| (u,v) \|_X \leq T$, then we get
\begin{gather*}
| J_\la (u,v) | \leq \big( \frac{1}{2}+\frac{\la}{2\la_1} \big)T^2+\eps C^2T^2+C_{\eps,q}T^q.	
\end{gather*}	
This implies that, $J_\la$ maps bounded sets of $X$ into bounded sets of $\R$, bearing out the proof of the lemma.
\end{proof}
\begin{lemma}\label{palais}
Let $J_\lambda$ be defined above in \eqref{eq:J}. Then, $J_\la$ satisfies (PS)$^*$ condition.
\end{lemma}
\begin{proof}
Let $(u_{\ba n}, v_{\ba n}) \subset X$ be a sequence with $\{\ba_n\}$ be admissible (defined in Section (\ref{abstract}) satisfies
\begin{gather}\label{Eq:J'}
	(u_{\ba_n}, v_{\ba_n}) \subset X_{\ba_n},\, \, c=sup_{n \in \N} J_\la(u_{\ba_n}, v_{\ba_n})<\infty, \, J'_\la(u_{\ba_n}, v_{\ba_n}) \to 0 \,\text{ in }\, X'.
\end{gather}	
We need to show that there exists $(u,v) \in X$ up to a subsequence such that,
\begin{gather*}
\| (u_{\ba_n}, v_{\ba_n})-(u,v)\|_X \to 0 \, \text{as}\, n \to \infty.	
\end{gather*}	
\textit{\bf Claim}. The sequence $(u_{\ba_n}, v_{\ba_n})$ is bounded in $X$.

\textit{Proof of Claim}. First, by Hypothesis $(H_7)$, we note that there exists $C_1,C_2>0$ such that
\begin{equation}\label{Eq:F}
F_u(x,u,v)u+F_v(x,u,v)v-2F(x,u,v) \geq C_1 \big( |u|^2+|v|^2 \big)^{\frac{p}{2}}-C_2,
\end{equation}
for all $(u,v) \in \R^2$, a.e. $x \in \Om$.
We suppose that there exists a subsequence of $(u_{\ba_n}, v_{\ba_n})$, still denoted by $(u_{\ba_n}, v_{\ba_n})$ such that $\|(u_{\ba_n}, v_{\ba_n})\|_X \to \infty$ as $n \to \infty$.
We compute using (\ref{Eq:F}) that,
\begin{equation}
\begin{aligned}
2J_\la(u_{\ba_n}, v_{\ba_n})&-\<J'_\la(u_{\ba_n}, v_{\ba_n}, (u_{\ba_n}, v_{\ba_n}))\>\\
&=\Iom \bigg( F_u(x, u_{\ba_n}, v_{\ba_n}) u_{\ba_n} +F_v(x, u_{\ba_n}, v_{\ba_n})v_{\ba_n} 
-2F(x,u_{\ba_n}, v_{\ba_n}) \bigg)\, dx\\
&\geq C_1 \Iom \big( |u_{\ba_n}|^2+ |v_{\ba_n}|^2   \big)^{\frac{p}{2}}\, dx-C_2 \mathcal{L}^N(\Om),
\end{aligned}
\end{equation}
where $\mathcal{L}^N(\Om)$ denotes the $N$-dimensional Lebesgue measure of $\Om$.
Using (\ref{Eq:J'}), we obtain,
\begin{equation*}
\begin{aligned}
&C_1 \frac{1}{\| u_{\ba_n}, v_{\ba_n}\|_X}\Iom \big( |u_{\ba_n}|^2+ |v_{\ba_n}|^2 \big)^{\frac{p}{2}}-C_2\frac{\mathcal{L}^N(\Om)}{ \| (u_{\ba_n}, v_{\ba_n})\|_X }\\
&\quad \quad\leq \frac{2C}{\| u_{\ba_n}, v_{\ba_n}\|_X}
-\frac{\<J'(u_{\ba_n}, v_{\ba_n}), (u_{\ba_n}, v_{\ba_n})\>}{\|u_{\ba_n}, v_{\ba_n}\|_X} \to 0 \, \text{as}\, n \to \infty.
\end{aligned}
\end{equation*}
Hence, we have, 
\begin{gather}\label{Eq:*}
\frac{\Iom \big( |u_{\ba_n}|^p +|v_{\ba_n}|^p  \big) \, dx}{ \| (u_{\ba_n}, v_{\ba_n})\|_X}	\to 0 \,\text{as}\, n \to \infty.
\end{gather}	
Let us use the decomposition of $(u_{\ba_n}, v_{\ba_n}) \in X$ as
\begin{gather*}
(u_{\ba_n}, v_{\ba_n})=(\bar{u}_{\ba_n}, \bar{v}_{\ba_n}) + (u_{\ba_n}^\perp, v_{\ba_n}^{\perp}) \in W_k \oplus W_k^\perp.
\end{gather*}
Using hypothesis \textit{$(H_3)$}, estimate (\ref{eq:W_k-}) and H\"older's inequality, we demonstrate that,
\begin{equation*}
\begin{aligned}
&\<J'_\la (u_{\ba_n}, v_{\ba_n})	, (u_{\ba_n}^\perp, v_{\ba_n}^\perp)\> \\
&=\< (u_{\ba_n}, v_{\ba_n}), (u_{\ba_n}^\perp, v_{\ba_n}^\perp) \>_X-\la \< L(u_{\ba_n}, v_{\ba_n}) , (u_{\ba_n}^\perp, v_{\ba_n}^\perp)  \>\\
&\qquad -\Iom \big[ F_u(x,u_{\ba_n}, v_{\ba_n})u_{\ba_n}^\perp+F_v(x, u_{\ba_n}, v_{\ba_n}) v_{\ba_n}^\perp  \big]\, dx\\
&\geq \| (u_{\ba_n}^\perp, v_{\ba_n}^\perp) \|_X^2-\frac{\la}{\la_{k+1}}\| (u_{\ba_n}^\perp, v_{\ba_n}^\perp )\|_X^2\\
&\qquad-\Iom \big[ |F_u(x,u_{\ba_n}, v_{\ba_n})|^2+|F_v(x,u_{\ba_n}, v_{\ba_n})|^2   \big]^{\frac{1}{2}} \big( |u_{\ba_n}^\perp|^2+ |v_{\ba_n}^\perp|^2  \big)^{\frac{1}{2}}\, dx\\
&\geq \big(1-\frac{\la}{\la_{k+1}}\big)\| (u_{\ba_n}^\perp, v_{\ba_n}^\perp)\|_X^2-C_2 \big(|u_{\ba_n}^\perp|_1+ |v_{\ba_n}^\perp|_1  \big)\\
&\qquad -C_3 \big( |u_{\ba_n}|_p^{q-1}+ |v_{\ba_n}|_p^{q-1} \big) \big( \Iom |u_{\ba_n}^\perp|^{\frac{p}{p-q+1}}+ |v_{\ba_n}^\perp|^{\frac{p}{p-q+1}}   \big)^{\frac{p-q+1}{p}}\\
&\geq \big(1-\frac{\la}{\la_{k+1}}\big)\| (u_{\ba_n}^\perp, v_{\ba_n}^\perp)\|_X^2-C_4 \| u_{\ba_n}^\perp, v_{\ba_n}^\perp\|_X-C_5 \big( |u_{\ba_n}|_p^{q-1}+ |v_{\ba_n}|_p^{q-1} \big)
\| ( u_{\ba_n}^\perp, v_{\ba_n}^\perp ) \|_X.
\end{aligned}
\end{equation*}		
This yields us,
\begin{gather}\label{Eq:*-}
	\big(1-\frac{\la}{\la_{k+1}}\big)\| (u_{\ba_n}^\perp, v_{\ba_n}^\perp)\|_X\leq \frac{\< J'_\la(u_{\ba_n}, v_{\ba_n}), (u_{\ba_n}^\perp, v_{\ba_n}^\perp)\>}{\| (u_{\ba_n}^\perp, v_{\ba_n}^\perp)\|_X}+C_4+C_5\big(|u_{\ba_n}|_p^{q-1}+ |v_{\ba_n}|_p^{q-1} \big).
\end{gather}	
As $1<q-1<p$, by Young's inequality we obtain,
\begin{gather}\label{Eq:*1}
\big( |u_{\ba_n}|_p^{q-1}+ |v_{\ba_n}|_p^{q-1} \big) \leq \frac{p}{q-1} \big( |u_{\ba_n}|_p^p+ |v_{\ba_n}|_p^p  \big)+(\frac{p-q-1}{p}).
\end{gather}
	Using (\ref{Eq:*1}), (\ref{Eq:*}), we declare from (\ref{Eq:*-}),
	\begin{gather}\label{Eq:-i}
	\frac{ \| (u_{\ba_n}^\perp, v_{\ba_n}^\perp)\|_X}{\| (u_{\ba_n}, v_{\ba_n})\|_X} \to 0, \text{as}\, n \to \infty.
	\end{gather}
In a similar technique, we justify that,
	\begin{gather}\label{Eq:-ii}
	\frac{ \| (\bar{u}_{\ba_n}, \bar{v}_{\ba_n})\|_X}{\| (u_{\ba_n}, v_{\ba_n})\|_X} \to 0, \text{as}\, n \to \infty.
\end{gather}
Combining (\ref{Eq:-i}) and (\ref{Eq:-ii}), we note that,
\begin{gather*}
1=\frac{\| (u_{\ba_n}, v_{\ba_n})\|_X^2}{\| ( u_{\ba_n}, v_{\ba_n})\|_X^2}=\frac{ \| (\bar{u}_{\ba_n}, \bar{v}_{\ba_n})_X^2 + \| (  u_{\ba_n}^\perp, v_{\ba_n}^\perp) \|_X^2 }{\| (u_{\ba_n}, v_{\ba_n})\|_X^2} \to 0 \,\text{as}\, n \to \infty,
\end{gather*}
which is not possible.
 Hence, $(u_{\ba_n}, v_{\ba_n})$ is bounded in $X$. This winds-up the Claim.

As $X$ is a Hilbert space, there exists $(u,v) \in X$ such that, up to a subsequence 
\begin{equation}\label{Eq:*2}
(u_{\ba_n}, v_{\ba_n}) \deb (u,v) \,\text{weakly in}\, X,\text{as}\, n \to \infty.
\end{equation}
We find that, 
\begin{equation*}
	(u_{\ba_n}, v_{\ba_n}) \to (u,v) \, \text{strongly in}\, L^2(\Om) \times L^2(\Om)
\end{equation*}	
and $u_{\ba_n}(x) \to u(x)$ a.e. $x \in \Rn$, $v_{\ba_n}(x) \to v(x)$ a.e. $x \in \Rn$.
By(\ref{Cpt-embed}), it turns out that,
\begin{equation}\label{Eq:*3}
\begin{aligned}
\|(u_{\ba_n}, v_{\ba_n})-(u,v) \|^2_X&=\< J'_\la(u_{\ba_n}, v_{\ba_n}), J'_\la(u,v), (u_{\ba_n}, v_{\ba_n})-(u,v)\>\\
&\quad+ \la \< L(u_{\ba_n}, v_{\ba_n})-L(u,v), (u_{\ba_n}, v_{\ba_n})-(u,v)\>\\
&\quad +\Iom \big( \big( F_u(x,u_{\ba_n}, v_{\ba_n})-F_u(x,u,v) \big)(u_{\ba_n}-u)\\
&\quad+ \big( F_v(x,u_{\ba_n}, v_{\ba_n})-F_u(x,u,v) \big) (v_{\ba_n}-v)  \big)\, dx.
\end{aligned}
\end{equation}
By using (\ref{Eq:*2}), we certify that,
\begin{gather*}
\<J'_\la(u_{\ba_n}, v_{\ba_n})-J'_\la(u,v), (u_{\ba_n}, v_{\ba_n})-(u,v)\> \to 0\,\text{ as }\, n \to \infty.
\end{gather*}
Analyzing the fact that $L$ is a compact linear operator, we see that,
\begin{gather*}
	\<L(u_{\ba_n}, v_{\ba_n})-L(u,v), (u_{\ba_n}, v_{\ba_n})-(u,v)\> \to 0\,\text{ as }\, n \to \infty.
\end{gather*}
By using the assumption \textit{$(H_3)$}, we infer that,
\begin{gather*}
\Iom \big[ (F_u(x,u_{\ba_n}, v_{\ba_n})-F_u(x,u,v)) (u_{\ba_n}-u)+ ( F_v(x,u_{\ba_n}, v_{\ba_n})-F_v(x,u,v)) (v_{\ba_n}-v)   \big]\,dx 
\to 0,	
\end{gather*}	
as $n \to \infty$. From (\ref{Eq:*3}), finally, we conclude that,
$$ \| (u_{\ba_n}, v_{\ba_n})-(u,v)\| \to 0 \,\text{as}\, n \to \infty.
$$
This summits the proof of this lemma.
\end{proof}

\begin{lemma}\label{coercive.1}
Let $J_\lambda$ be defined above in \eqref{eq:J}. Then, there holds
$$J_\la(u_k,v_k) \to -\infty \,\text{for}\, (u_k,v_k) \in X_k^\perp \oplus X^2 \, \text{ with }\, \| (u_k,v_k)\|_X \to \infty.
$$
\end{lemma}
\begin{proof}
Considering as $X_k^1 \oplus X^2$ is a finite dimensional subspace of $X$,  any norms on $X_k^1 \oplus X^2$ are equivalent. Hence, there exists $C_6>0$ such that
\begin{gather}\label{Eq:*4}
\| (u, v)\|_X \leq C_6 \big( |u|_2+|u|_{2} \big) \,\text{for all}\, (u,v) \in X_k^1 \oplus X^2.	
\end{gather}	
From hypothsis \textit{$(H_6)$}, we infer that for any $T>0$, there exists constant $C_T>0$ such that,
\begin{gather}\label{Eq:*5}
F(x,u,v) \geq T \big( |u|^2+|v|^2  \big)-C_T, \text{for all}\, (u,v) \in \R^2, \,\text{a.e.} \, x \in \Om.	
\end{gather}	
Using (\ref{eq:W_k}), (\ref{Eq:*4}), (\ref{Eq:*5}) and for $T>0$ large enough and for $$(u_k, v_k)=(\bar{u}_k, \bar{v}_k) \oplus (\tilde{u}_k,\tilde{v}_k) \in X_k^1 \oplus X^2,$$ we have,
\begin{equation*}
\begin{aligned}
J_\la(u_k,v_k)&\leq \big(\frac{1}{2}-\frac{\la}{2\la_k} \big)\| (\tilde{u}_k,\tilde{v}_k) \|_X^2+\frac{1}{2}
\| (\bar{u}_k,\bar{v}_k)\|_X^2-\frac{\la}{2}\< L((\bar{u}_k,\bar{v}_k)), (\bar{u}_k,\bar{v}_k)\>\\
&\quad-\Iom \big(T (|{u}_k|^2+|{v}_k|^2) -C_T \big)\, dx\\
&\leq \big(\frac{1}{2}-\frac{\la}{2\la_k} \big)\| (\tilde{u}_k,\tilde{v}_k) \|_X^2+\big( \frac{1}{2}+C_7   \big)\|(\bar{u}_k,\bar{v}_k)\|_X^2\\
&\quad -TC_8\big( \|(\bar{u}_k,\bar{v}_k)\|_X^2+\| (\bar{u}_k,\bar{v}_k)\|_X^2 \big)+C_T\mathcal{L}^N(\Om)\\
& \leq \big(\frac{1}{2}-\frac{\la}{2\la_k} \big)\| ({u}_k,{v}_k) \|_X^2+C_T \mathcal{L}^N(\Om).
\end{aligned}	
\end{equation*}	
This essentially implies, $J_\la(u_k,v_k) \to -\infty$ as $(u_k,v_k) \in X_k^1 \oplus X^2$ with $\|(u_k,v_k)\|_X \to \infty$, finishing the proof.
\end{proof}
\begin{proof}[Proof of Theorem \ref{thm-1}]
Combining Lemma \ref{c1}-\ref{coercive.1}, we conclude that $J_\la$ satisfies the conditions (a)-(d) of Theorem \ref{assthm.1}. Therefore, problem $(\mathcal{P}_\la)$ has a solution for $\la_k<\la<\la_{k+1}$ in $X$. This concludes the proof of the theorem.
\end{proof}

\section{Proof of Theorem \ref*{thm-2}}\label{Thm-2}
In this section, we toss around for the case $\la=\la_1$ and deal with the functional 
$J_{\la_1}:X \to \R$ given by:
\begin{equation}\label{eq:J-la_1}
J_{\la_1}(u,v)=\frac{1}{2} \|(u,v)\|_X^2-\frac{\la_1}{2}\< L(u,v), (u,v) \>-\Iom F(x,u,v)\, dx, \, (u,v) \in X,
\end{equation}	
and  define $K:X \to \R$ given by,
\begin{gather}\label{eq:K}
K(u,v)=\Iom F(x,u,v)\, dx \,\quad (u,v) \in X.
\end{gather}	
We recall from Lemma \ref{c1} that $K$ is fo class $C^1$  and we denote its derivative by $K'.$
\begin{lemma}\label{compact}
Let $K$ be defined in \eqref{eq:K} above. Then $K':X \to X'$ is compact.
\end{lemma}
\begin{proof}
Let $\{(u_n,v_n) \} \subset X$ be a bounded sequence in $X$. We need to show that upto a subsequence $K'(u_n,v_n)$ is convergent in $X'$. As $(u_n,v_n)$ is bounded in $X$, there exists $(u_0,v_0) \in X$ such that up to a subsequence,
$(u_n,v_n) \to (u_0,v_0)$ strongly in $L^2(\Om) \times L^2(\Om)$, $u_n \to u$ a.e. in $\Rn$ and $v_n \to v$ a.e. in $\Rn$. By using Hypothesis \textit{$(H4)$}, we access that for every $\eps>0$, there exists $d_\eps>0$ such that
\begin{equation}\label{Eq:F-}
|F_u(x,u,v)|+|F_v(x,u,v)| \leq (a(x) + \eps) ( |u|^2+|v|^2 )^{\frac{1}{2}}+d_\epsilon,
\end{equation}	
a.e. $x \in \Om$, for all $(u,v) \in \R^2$.
In view of the fact that $a \in L^{\infty}(\Om)$ and [Lemma 2.3, \cite{Drabek-94}], we attain that, 
\begin{gather}\label{eq:F-u,v}
	| F_u( \cdot, u_n, v_n)-F_u(\cdot,u,v)|_{2} + | F_v( \cdot, u_n, v_n)-F_v(\cdot,u,v)|_{2} \to 0, \text{ as }\, n \to \infty.
\end{gather}	
For $(w,z) \in X$, we notice that,
\begin{equation*}
	\begin{aligned}
	&\< K'(u_n,v_n)-K'(u_0,v_0), (w,z)\>\\&=\Iom \big[ (F_u(x, u_n, v_n)-F_u(x,u_0,v_0))w
	+  (F_v(x, u_n, v_n)-F_v(x,u_0,v_0))z  \big]\, dx.\\	
	\end{aligned}	
\end{equation*}	
Using H\"older's inequality, we fetch that,
\begin{equation*}
	\begin{aligned}
		&|\< K'(u_n,v_n)-K'(u_0,v_0), (w,z)\>\\
		&\qquad \leq C \big( |F_u(\cdot, u_n, v_n)-F_u(\cdot,u_0,v_0)|_{2}+ |F_v(\cdot, u_n, v_n)-F_v(\cdot,u_0,v_0)|_{2} \big) (\|(w,z) \|_X )
	\end{aligned}	
\end{equation*}	
This yields, using (\ref{eq:F-u,v}) that,
\begin{equation*}
\|K'(u_n,v_n)-K'(u_0,v_0)\|_{X'}	\to 0 \, \text{ as }\, n \to \infty.
\end{equation*}	
This proves the lemma.
\end{proof}

We can write from \eqref{eq:J-la_1},
\begin{equation}\label{eq:+}
\begin{aligned}
	J_{\la_1}(u,v)&=\frac{1}{2}\< (u,v), (u,v)\>_X-\frac{\la_1}{2}\< \mathcal{D}^{-1} L(u,v), (u,v)\>-K(u,v)\\
	&=\frac{1}{2} \< (I-\la_1\mathcal{D}^{-1} L)(u,v), (u,v)\>_X-K(u,v), \, \text{for each}\, (u,v) \in X.
\end{aligned}	
\end{equation}
We take into account the following decomposition
\begin{gather*}
(u,v)=(\bar{u},\bar{v}) + ( u^{\perp},v^{\perp})	\in W_1 \oplus W_1^\perp.
\end{gather*}	
Using the problem $(\mathcal{E}_\la)$, we note that both $W_1$ and $W_1^\perp$ are invariant under $(I-\la_1 \mathcal{D}^{-1}L)$. Therefore, we gather that,
\begin{equation}\label{eq:++}
(I-\la_1 \mathcal{D}^{-1} L)(u,v)=(I-\la_1\mathcal{D}^{-1} L)(\bar{u}, \bar{v})+ (I-\la_1 \mathcal{D}^{-1}L)({u}^\perp, {v}^\perp), \text{ for all}\, (u,v) \in X.
\end{equation}	
Hence, from  equations (\ref{eq:+}) and (\ref{eq:++}), we accomplish
 that $J_{\la_1}$ satisfies the condition 
{{Theorem \ref{assthm.1}(a)} and thereby earning the proof of the lemma.

\begin{lemma}\label{cerami}
The functional $J_\la$ defined in  \eqref{eq:J-la_1} satisfies Cerami condition. 
\end{lemma}
\begin{proof}
Let  for $b \in \R$, $\{(u_n, v_n)\}_{n \geq 1} \subset X$ be a sequence such that
\begin{equation}\label{Eq:J_la}
J_{\la_1}(u_n,v_n) \to b \,\text{and}\, \big( 1+ \| (u_n,v_n)\|_X \big) \| J'_{\la_1}(u_n,v_n)\|_{X'} \to 0, \text{as}\, n \to \infty.
\end{equation}
We need to show that there exists $(u_0,v_0) \in X$ such that up to a subsequence,
\begin{gather*}
\| (u_n,v_n)-(u_0,v_0) \|_X \to 0 \,\text{as}\, n \to \infty.
\end{gather*}
\textit{Claim:} $(u_n,v_n)$ is bounded sequence in $X$.

We achieve the proof by the method of contradiction.
Indeed, there exists a subsequence which is again denoted by $(u_n,v_n)$
such that $\|(u_n,v_n)\|_X \to 
\infty $ as $n \to \infty$. Let us write,
\begin{gather*}
( \tilde{u}_n, \tilde{v}_n)= \frac{(u_n,v_n)}{\| (u_n,v_n)\|_X},\,	\, n \in \N.
\end{gather*}	
Hence, $\|(\tilde{u}_n, \tilde{v}_n)\|_X=1 $, that is, $(\tilde{u}_n, \tilde{v}_n)$ is a bounded sequence, which witnesses that there exists $(\tilde{u}_0, \tilde{v}_0) \in X$ such that up to a subsequence,
$(\tilde{u}_n, \tilde{v}_n) \deb (\tilde{u}_0, \tilde{v}_0) $ in $X$, $(\tilde{u}_n, \tilde{v}_n) \to \tilde{u}_0, \tilde{v}_0$ strongly in $L^2(\Om) \times L^2(\Om)$, $\tilde{u}_n \to \tilde{u}_0$ a.e.in $\Rn$ and $\tilde{v}_n \to  \tilde{v}_0$ a.e. in $\Rn$. By Hypothesis \textit{$(H_8)$}, there exists $C_7,C_8>0$ such that
\begin{gather}\label{eq:&}
	2F(x,u,v)-(F_u(x,u,v)u+F_v(x,u,v)v) \geq C_7 \big( |u|^2+|v|^2 \big)^{\frac{\ga}{2}}-C_8,
\end{gather}	
for a.e. $x \in \Om$, for all $(u,v) \in \R^2$.
Therefore, we get,
$$ \<J'_{\la_1}(u_n,v_n), (u_n,v_n)\>-2 J_{\la_1}(u_n,v_n) \geq C_7 \Iom (|u_n|^2+|v_n|^2 )^{\frac
{\ga}{2}}-C_8 |\Om|.
$$
As before, we infer that, 
\begin{equation}\label{*2}
\frac{1}{\| (u_n,v_n)\|_X} \Iom ( |u_n|^2+|v_n|^2 )^{\frac{\ga}{2}} \, dx \to 0 \,\text{as}\, n \to \infty.
\end{equation}
Let us monitor the decomposition 
\begin{gather*}
	(u_n,v_n)=(\bar{u}_n, \bar{v}_n) + (u_n^\perp,v_n^\perp) \in W_1 \oplus W_1^\perp. 
\end{gather*}	
Using the embedding $X \hookrightarrow L^r(\Om)$ for $1 \leq r \leq 2^*_s$, estimate (\ref{eq:W_k-}) and H\"older's inequality,
we note that,
\begin{equation*}
	\begin{aligned}
	&\<J_{\la_1}(u_n,v_n), (u_n^\perp,v_n^\perp)\>\\
	&\geq (1-\frac{\la_1}{\la_2})	\| (u_n^\perp,v_n^\perp) \|_X^2-\Iom \bigg[ |F_u(x, u_n,v_n)||u_n^\perp| + |F_v(x,u_n,v_n)| |v_n^\perp|  \bigg]	\, dx\\
	&\geq (1-\frac{\la_1}{\la_2})	\| (u_n^\perp,v_n^\perp) \|_X^2-d_\eps 
	\big( |u_n^\perp|_{1}+|v_n^\perp|_{1}  \big)\\
	&\quad-C_8 \big( |u_n|_{\ga}+|v_n|_{\ga}  \big)
	\big( |u_n^\perp|_{\frac{\ga}{\ga-1}}+|v_n^\perp|_{\frac{\ga}{\ga-1}}  \big)\\
	&\geq (1-\frac{\la_1}{\la_2})	\| (u_n^\perp,v_n^\perp) \|_X^2-d_\eps  C_9 \| (u_n^\perp,v_n^\perp)\|_X-C_{10}  \big( |u_n|_{\ga}+|v_n|_{\ga}  \big)	\| (u_n^\perp,v_n^\perp) \|_X.
	\end{aligned}	
\end{equation*}	
As a consequence, we get hold of the following,
\begin{equation*}
	\begin{aligned}
\big(1-\frac{\la_1}{\la_2} \big)	\| (u_n^\perp,v_n^\perp) \|_X &\leq \<J_{\la_1}(u_n,v_n), \frac{(u_n^\perp,v_n^\perp)}{\| (u_n^\perp,v_n^\perp)\|_X}\>+d_\eps C_9\\
&\qquad +\frac{C_{10}}{\ga} \big(|u_n|^{\ga}_{{\ga}}+ |v_n|^{\ga}_{{\ga}}\big)+C_{10}( \frac{\ga-1}{\ga}).
  \end{aligned}
\end{equation*}
Using (\ref{*2}) and (\ref{Eq:J_la}), we infer,
\begin{gather}\label{*1}
\frac{\| (u_n^\perp,v_n^\perp) \|_X}{\| (u_n,v_n) \|_X}	\to 0 \, \text{ as }\, n \to \infty.
\end{gather}	
In view of (\ref{*1}), we mark that,
\begin{equation*}
\begin{aligned}
	&\bigg| \frac{(u_n,v_n)}{\| (u_n,v_n) \|_X}-(\tilde{u}_0, \tilde{v}_0 )  \bigg|_{L^2}\\
	&=\bigg| (\tilde{u}_n, \tilde{v}_n)-(\tilde{u}_0, \tilde{v}_0)-\frac{({u}_n^\perp, {v}_n^\perp)}{\| (u_n,v_n)\|_X}\bigg|_{L^2}\\
	&\leq \big| (\tilde{u}_n, \tilde{v}_n)-(\tilde{u}_0, \tilde{v}_0)|_{L^2}+C\frac{\|({u}_n^\perp, {v}_n^\perp) \|_X}{\| (u_n,v_n)\|_X} \to 0, \,\text{as}\, n \to \infty. \\
\end{aligned}		
\end{equation*}	
As $W_1$ is a finite dimensional subspace of $X$ and $(\bar{u}_n, \bar{v}_n) \in W_1$ for all $n \in \N$, so
\begin{equation*}
\frac{(u_n,v_n)}{\| (u_n,v_n)\|_X} \to (\tilde{u}_0, \tilde{v}_0) \, \text{ in }\, W_1.	
\end{equation*}	
Consequently since  any two norms are equivalent in $W_1$, using (\ref{*1}), we observe, 
\begin{equation*}
\| 	(\tilde{u}_n, \tilde{v}_n)-(\tilde{u}_0, \tilde{v}_0) \|_X
\leq \Big\| \frac{({u}_n, {v}_n)}{\| ({u}_n, {v}_n)\|_X}-(\tilde{u}_0, \tilde{v}_0) \Big\|_X+ \frac{\| ({u}_n^\perp, {v}_n^\perp)\|_X}{\| (u_n,v_n)\|_X} \to 0,
\end{equation*}	
as $n \to \infty$. Hence, $(\tilde{u}_n, \tilde{v}_n) \to (\tilde{u}_0, \tilde{v}_0)$  in $X$ and since $\| (\tilde{u}_n, \tilde{v}_n) \|_X=1$, for all $n \in \N$, $\| (\tilde{u}_0, \tilde{v}_0) \|_X=1$. This yields us, 
$$\big| \{ y \in \Om: (\tilde{u}_0, \tilde{v}_0)(y) \neq (0,0) \} \big| >0.$$
Using (\ref{eq:&}), we make out that,
\begin{gather}
\lim_{n \to \infty} \int_{ \{ x \in \Om: (\tilde{u}_0, \tilde{v}_0)(x) \neq (0,0)  \}}\Big(
2F(x,u_n,v_n)-\big( F_u(x,u_n,v_n)u_n+F_v(x,u_n,v_n)v_n) \Big)\, dx=\infty.
\end{gather}	
On the other way, we deduce that,
\begin{equation*}
	\begin{aligned}
		&\lim_{n \to \infty} \int_{ \{ x \in \Om: (\tilde{u}_0, \tilde{v}_0)(x) \neq (0,0)  \}}(
		2F(x,u_n,v_n)-\big( F_u(x,u_n,v_n)u_n+F_v(x,u_n,v_n)v_n) \big)\, dx\\
		&\leq \lim_{n \to \infty}\Iom \big| 	2F(x,u_n,v_n)-\big( F_u(x,u_n,v_n)u_n+F_v(x,u_n,v_n)v_n)\big|\, dx\\
		&\leq \lim_{n \to \infty} | \< J'_{\la_1}(u_n,v_n), (u_n,v_n)\>-2J_\la(u_n,v_n)\big| \leq 2C_{11}.
	\end{aligned}
\end{equation*}	
This gives us a contradiction to (\ref{*1}), deducing $\{ (u_n,v_n)\}$ is a bounded sequence in $X$. So, there 
exists $(u,v) \in X$, up to a subsequence, we get,
$$ (u_n,v_n) \deb (u,v) \,\text{weakly in}\, X. 
$$
Proceeding similarly, as in the Lemma \ref{palais}, we infer that,
$$ \| (u_n,v_n)-(u,v)\|_X \to 0\, \text{as}\, n \to \infty,
$$
thereby fetching the proof.
\end{proof}
\begin{lemma}\label{anti_coercive} Let $J_{\la_1}$ be defined in \eqref{eq:J-la_1}. Then, $J_{\la_1}$ is anti-coercive in $W_1$, that is, $J_\la(u,v) \to -\infty$ as $(u,v) \in W_1$ with $\| (u,v) \|_X \to \infty$.
\end{lemma}

\begin{proof}
Let $(u_0,v_0) \in \R^2$ with $|u_0|^2+|v_0|^2=1$. Let $\eps>0$ be arbitrary. Then, by Hypothesis \textit{$(H_9)$}, there exists $M_\eps$ such that,
\begin{gather}\label{F}
	\frac{F_u(x,u,v)u+F_v(x,u,v)v}{|u|^2+|v|^2} \geq -\eps, \text{ for all}\,( |u|^2+|v|^2)^{\frac{1}{2}} \geq M_\eps.
\end{gather}	
We then calculate for $l>M_\eps$, using (\ref{F}), that
\begin{equation*}
	\begin{aligned}
	F(x,lu_0,lv_0)&=F(x, M_\eps x_0, M_\eps v_0)+\int_{M_\eps}^{l} \frac{d}{dt}(F(x,tu_0,tv_0))\, dt\\
	&\geq F(x, M_\eps x_0, M_\eps v_0)+\int_{M_\eps}^l \frac{ \big( F_u(x,tu_0,tv_0)tu_0+F_v(x,tu_0,tv_0)tv_0  \big)  }{t}\, dt\\
		&\geq F(x, M_\eps x_0, M_\eps v_0)-\eps \int_{M_\eps}^lt \, dt\\
		&\geq F(x, M_\eps x_0, M_\eps v_0)-\frac{\eps}{2}l^2.
	\end{aligned}	
\end{equation*}	
This implies,
\begin{gather*}
	\frac{F(x,lu_0,lv_0)}{l^2} \geq 	\frac{F(x,M_\eps u_0,M_\eps v_0)}{l^2}-\frac{\eps}{2},
\end{gather*}	
for all $l>M_\eps$, thereupon letting $l \to \infty$ and since $\eps>0$ is arbitrary,
 we see that,
\begin{gather*}
	\liminf_{l \to \infty} \frac{F(x,lu_0,lv_0)}{l^2} \geq 0.
\end{gather*}	
On the other hand, for any $r \in \R$,
\begin{equation}\label{F_i}
\begin{aligned}
F_u(x, ru_0, rv_0)r u_0&+ F_v(x, ru_0, rv_0)rv_0-2F(x, ru_0, rv_0)\\
&\quad \leq C_8-C_7 \big( |r u_0|^2+|r v_0|^2  \big)^{\frac{\ga}{2}}=C_8-C_7|r|^{\ga}.
\end{aligned}
\end{equation}
For any $T>0$, we take,
$$ |r| \geq \big( \frac{C_8+T}{C_7}  \big)^{\frac{1}{\ga}}:=C_T.
$$
Thus, $C_7|r|^{\ga}-C_8 \geq T$. Hence, from (\ref{F_i}), we obtain,
\begin{gather*}
	F_u(x, ru_0, rv_0)r u_0+ F_v(x, ru_0, rv_0)rv_0-2F(x, ru_0, rv_0) \leq -T,\,\text{for all } |r| \geq C_T.
\end{gather*}	
Consequently, for $C_T \leq r<l$,
\begin{gather*}
\frac{F(x, ru_0, rv_0)}{r^2} \geq 	\frac{F(x, l u_0, l v_0)}{l^2}+\frac{T}{2} \big( \frac{1}{s^2}-\frac{1}{l^2} \big).
\end{gather*}	
On letting $l \to \infty$, it is elementary to see that,
\begin{gather*}
	F(x, ru_0, rv_0) \geq \frac{T}{2} \,\text{ a.e.}\, x \in \Om.
\end{gather*}	
In a similar manner, for $s \leq -C_T$,
$$ F(x, ru_0, rv_0) \geq \frac{T}{2}\,\text{ a.e.}\, x \in \Om.
$$
Hence, for any $T>0$, there exists $C_T>0$ such that,
\begin{gather*}
	F(x, ru_0, rv_0) \geq \frac{T}{2} \,\text{ for all }\, |s| \geq C_T \,\text{ and a.e. }\, x \in \Om.
\end{gather*}	
As a result, we obtain,
$$ J_{\la_1}(u,v) \leq -K(u,v)=-\Iom F(x,u,v)\,dx \leq -\frac{T}{2}\mathcal{L}^N(\Om),
$$
for $(u,v) \in W_1$ with $\| (u,v)\|_X \geq C_T$. This signifies,
$J_{\la_1}(u,v) \to -\infty$ as $(u,v) \in W_1$ with $\| (u,v)\|_X \to \infty$,  concluding the proof of the lemma.
\end{proof}

\begin{lemma}\label{coercive.2}
$J_{\la_1}$ is coercive in $W_1^\perp$, that is,
$$ J_{\la_1}(u,v) \to \infty \,\text{as}\, (u,v) \in W_1^\perp \,\text{ with }\,
\|(u,v)\|_X \to \infty.
$$
\end{lemma}
\begin{proof}
 We first establish the following Claim.\\
\textit{\bf Claim}. There exists $\tilde{\la}>0$ such that
\begin{equation}\label{Eq:}
\begin{aligned}
\frac{1}{2}\| (u,v)\|_X^2-\frac{\la_1}{2} \<L(u,v), (u,v)\>-\frac{1}{2}\Iom a(x) \big( |u|^2+|v|^2 \big)\,dx
\geq \tilde{\la} \|(u,v)\|_X^2 \,\text{ for all } (u,v) \in W_1^\perp.
\end{aligned}
\end{equation}
We prove by the method of contradiction. This imports, for each $n \in \N$, there exists $(u_n,v_n) \in X$ such that $\| (u_n,v_n)\|_X=1$, but
\begin{equation}\label{eq:n}
\frac{1}{2}\| (u_n,v_n)\|_X^2-\frac{\la_1}{2}\< L(u_n,v_n), (u_n,v_n)\>-\frac{1}{2}\Iom a(x) \big( |u_n|^2+|v_n|^2 \big)\,dx<\frac{1}{n}.
\end{equation}
As $(u_n,v_n)$ is bounded in $X$, there exists $(u_0,v_0) \in W_1^\perp$ such that up to a subsequence,
$(u_n, v_n) \deb (u_0,v_0)$ weakly in $W_1^\perp$ and $(u_n, v_n) \to (u_0,v_0)$ strongly in $L^2(\Om) \times L^2(\Om)$. By the compactness of the operator $L$ and weak lower semicontinuity of the norm $\|\cdot\|_X$, it is evident that,
\begin{equation*}
	\begin{aligned}
&\frac{1}{2}\| (u_0,v_0)\|_X^2-\frac{\la_1}{2}\< L(u_0,v_0), (u_0,v_0)\>-\frac{1}{2}\Iom a(x) \big( |u_0|^2+|v_0|^2 \big)\,dx\\
&\leq \liminf_{n \to \infty} \Big( \frac{1}{2}\| (u_n,v_n)\|_X^2-\frac{\la_1}{2}\< L(u_n,v_n), (u_n,v_n)\>-\frac{1}{2}\Iom a(x) \big( |u_n|^2+|v_n|^2 \big)\,dx    \Big).
	\end{aligned}
\end{equation*}	
Subsequently, we verify,
\begin{gather*}
	\frac{1}{2}\| (u_0,v_0)\|_X^2-\frac{\la_1}{2}\< L(u_0,v_0), (u_0,v_0)\>
	\leq \frac{1}{2}\Iom a(x) \big( |u_0|^2+|v_0|^2 \big)\,dx.
\end{gather*}	
As $(u_0,v_0) \in W_1^\perp,$ using the hypothesis
$g(x) \equiv 0$ and $\min_{x \in \bar{\Om}}\{f(x), h(x)\} \geq 1$, we infer that,
\begin{gather*}
	\frac{1}{2} ( \la_2-\la_1) \<L(u_0,v_0), (u_0,v_0)\> \leq \frac{1}{2}\Iom a(x) ( |u_0|^2+|v_0|^2 )\, dx.
\end{gather*}	
This implies,
$$ (\la_2-\la_1) \Iom (|u_0|^2+|v_0|^2) \leq \Iom a(x) (|u_0|^2+|v_0|^2)\, dx,
$$
which in turn yields, since $a(x) \leq \la_2-\la_1$ a.e. $x \in \Om$,
$$  0 \leq \Iom \big[ a(x)-(\la_2-\la_1) \big] (|u_0|^2+|v_0|^2)\, dx \leq 0.$$
As $a(x)<\la_2-\la_1$ on a subset of $\Om$ of positive measure, we have, $u_0(x)=v_0(x)=0$ on $\{ x: a(x) \leq \la_2-\la_1 \}$.
By strong maximum principle, it follows that, $u_0 \equiv 0 \equiv v_0$ a.e. in $\Om$. However, from (\ref{eq:n}), it is obvious that,
\begin{gather*}
	\la_1 \< L(u_0,v_0), (u_0,v_0)\>_X+\Iom a(x) \big( |u_0|^2+|v_0|^2 \big)\, dx=1.
\end{gather*}	
This is a contradiction to the fact that $u_0 \equiv 0 \equiv v_0$ a.e. in $\Om$. Hence, (\ref{Eq:}) holds true, thereby justifying the Claim. 

Using embedding \ref{Sob-embed} together with (\ref{Eq:F-})and (\ref{Eq:}) yields for all $(u,v) \in W_1^\perp$ and for $\eps>0$ small,
\begin{equation*}
\begin{aligned}
J_{\la_1}(u,v)&\geq \frac{1}{2}\| (u,v)\|_X^2-\frac{\la_1}{2}\< L(u,v), (u,v)\>-\frac{1}{2}\Iom a(x) ( |u|^2+|v|^2 )\, dx\\
&\quad -\frac{\eps}{2} \Iom (|u|^2+|v|^2)\, dx-\frac{d_\eps}{2}\Iom (|u|^2+|v|^2)^{\frac{1}{2}}\, dx-\Iom |F(x,0,0)|\,dx\\
&\geq ( \tilde{\la} -\frac{C\eps}{2}) \| (u,v)\|_X^2-C d_\eps \| (u,v)\|_X-\Iom |F(x,0,0)|\,dx.
\end{aligned}		
\end{equation*}	
Henceforth, choosing $\eps<\frac{2\tilde{\la}}{C}$, we note that,
$J_{\la_1}(u,v) \to \infty$ as $(u,v) \in W_1^\perp$ with $\|(u,v)\|_X \to \infty$, settling  the proof of this lemma.
\end{proof}
\begin{proof}[Proof of Theorem \ref{thm-2}]
Bringing together all the Lemmas \ref{compact}-\ref{coercive.2} above, we conclude that $J_{\la_1}$ satisfies all the conditions of {Theorem \ref{assthm.2}}, thereupon proving that problem $(\mathcal{P}_{\la_1})$ has a solution.
\end{proof}


\bibliography{biblio.bib}
\bibliographystyle{acm} 
\end{document}